\documentclass[11pt]{article}
\usepackage[a4paper,margin=3cm]{geometry}
\usepackage{amsmath}
\usepackage{amssymb}

\newtheorem{theorem}{Theorem}[section]
\newtheorem{lemma}[theorem]{Lemma}
\newtheorem{corollary}[theorem]{Corollary}
\newtheorem{proposition}[theorem]{Proposition}
\newtheorem{definition}[theorem]{Definition}
\newenvironment{proof}{{\bf Proof.}}{\hfill$\Box$\\}
\newenvironment{proof of}{{\bf Proof of}}{\hfill$\Box$\\}
\newenvironment{remark}{{\vskip 1ex\bf Remark.}}{\\}

\newcommand{\C}{\mathbb{C}}

\newcommand{\Q}{\mathbb{Q}}
\newcommand{\R}{\mathbb{R}}
\newcommand{\ad}{\mathrm{ad}}
\newcommand{\rk}{\mathrm{rk}}

\newcommand{\spn}{\mathrm{span}}

\newcommand{\VV}{\mathcal{V}}

\newcommand{\XX}{\mathcal{X}}
\newcommand{\FF}{\mathcal{F}}
\newcommand{\DD}{\mathcal{D}}

\title{{\bf Paraconformal structures, ordinary differential equations and totally geodesic manifolds}}
\author{
Wojciech Kry\'nski\thanks{
{\bf Institute of Mathematics, Polish Academy of Sciences, ul.~\'Sniadeckich 8, 00-956 Warszawa, Poland}\newline 
E-mail: krynski@impan.pl.}
}

\begin{document}
\maketitle
\begin{abstract}
We construct point invariants of ordinary differential equations that generalise the Cartan invariants of equations of order two and three. The vanishing of the invariants is equivalent to the existence of a totally geodesic paraconformal structure which consist of a paraconformal structure, an adapted $GL(2,\R)$-connection and a two-parameter family of totally geodesic hypersurfaces on the solution space. The structures coincide with the projective structures in dimension 2 and with the Einstein-Weyl structures of Lorentzian signature in dimension 3. We show that the totally geodesic paraconformal structures in higher dimensions can be described by a natural analogue of the Hitchin twistor construction.  We present a general example of Veronese webs which correspond to the hyper-CR Einstein-Weyl structures in dimension 3. The Veronese webs are described by a hierarchy of integrable systems.
\end{abstract}

\section{Introduction}
A paraconformal structure, or a $GL(2,\R)$-structure, on a manifold $M$ is a smooth field of rational normal curves in the tangent bundle $TM$. The structures have been investigated since the seminal paper of Bryant \cite{B} who has related the geometry of four-dimensional $GL(2,\R)$-structures to the contact geometry of ordinary differential equations (ODEs) of order four and consequently constructed examples of spaces with exotic holonomies. The result of Bryant can be seen as a generalisation of the paper of Chern \cite{Ch} who has proved that the conformal Lorentzian metrics on three-dimensional manifolds can be obtained from ODEs of order three (see also \cite{FKN}). The higher dimensional cases has been treated by many authors, for example in \cite{Db,DT,GN,N}. It is proved that the solution space of an ODE has a canonical paraconformal structure if and only if the W\"unschmann invariants vanish. 

In the present paper we consider paraconformal structures admitting the following additional structure: an adapted connection $\nabla$ and a 2-parameter family of hyper-surfaces totally geodesic with respect to $\nabla$. The structures will be referred to as the \emph{totally geodesic paraconformal structures}. The structures are very well known in low dimensions. Indeed, in dimension 2 the structures coincide with the projective structures \cite{BDE} and in dimension 3 the structures coincide with the Einstein-Weyl structures of Lorentzian signature \cite{D,T}. An unified approach to the projective structures on a plane and to the three-dimensional Einstein-Weyl structures was given in the complex setting by Hitchin \cite{H} in terms of a twistor construction. Much earlier, it was proved by E.~Cartan that in both cases the geometry is related to the point geometry of ODEs \cite{C1,C2}. The solution space of an ODE of order 2 or 3, respectively, has a canonical projective structure or an Einstein-Weyl structure, respectively, if and only if the Cartan invariant or the Cartan and the W\"unschmann invariants, respectively, vanish. The Cartan invariant of the second order ODEs has already been known to Tresse \cite{Kr,Tr}.

Our first aim in the present paper is to provide an unified approach to the Cartan invariants of second and third order ODEs given up to point transformations and generalise them to higher order ODEs. The second aim is to analyse the geometry of the totally geodesic paraconformal structures. Finally we consider a general example based on special families of foliations, called Veronese webs, introduced by Gelfand and Zakharevich \cite{GZ} in connection to bi-Hamiltonian structures on odd dimensional manifolds.

Our first new result is Theorem \ref{thm1b} that gives a characterisation of those paraconformal structures that can be constructed from ODEs. This result concerns the contact geometry of ODEs and, in a sense, completes results of \cite{Db,DT,GN}.
Sections \ref{sec_ODEcon}-\ref{sec_gen} concern point geometry of ODEs and are the core of the paper. In particular Theorems \ref{thm_ord2} and \ref{thm_ord3} provide new approach to the Cartan invariants of ODEs of order 2 and 3 and give new, more simple, formulae for the invariants.  Theorems \ref{thm_ord4} and \ref{thm_gen} generalise the Cartan invariants to higher dimensions. 

Section \ref{sec_twistor} is devoted to a natural generalisation of the Hitchin twistor construction. The Hitchin construction involves a two-dimensional manifold and a curve with a normal bundle $O(1)$ or $O(2)$.  Clearly one can consider curves with normal bundles $O(k)$, $k>2$. We argue that so-obtained structures correspond to higher-dimensional totally geodesic paraconformal structures. This should be compared to \cite{MP} where the authors are interested in torsion-free connections. On contrary, generic totally geodesic paraconformal structures considered in the present paper have non-trivial torsion. In the Hitchin's paper there is no construction of invariants on the side of ODEs. We concentrate on this issue in the present paper and our invariants characterise those equations for which the solutions are curves with self intersection number $k$.

Section \ref{sec_ricci} is devoted to the Ricci curvature tensor of a totally geodesic paraconformal connection. We prove that the symmetric part of the Ricci curvature tensor is a section of the bundle of symmetric 2-tensors annihilating all null directions of the structure. In dimension 3 the condition is equivalent to the Einstein-Weyl equation.

The last Section of the paper is devoted to Veronese webs.  We show that any Veronese web defines a totally geodesic paraconformal structure such that the associated twistor space fibres over $\R P^1$. In particular, Veronese webs in dimension 3 give an alternative description of the hyper-CR Einstein-Weyl structures \cite{D,DK1,DK2}. We prove that in the general case the Veronese webs, or equivalently the totally geodesic paraconformal structures such that the corresponding twistor space fibres over $\R P^1$, are in a one to one correspondence with the solutions to the system
\begin{equation}\label{eq_int_sys}
(a_i-a_j)\partial_0w\partial_i\partial_jw+ a_j\partial_iw\partial_j\partial_0w- a_i\partial_jw\partial _i\partial_0w=0,\qquad  i,j=1,\ldots,k,
\end{equation}
where $a_i$ are distinct constants and $w\colon\R^{k+1}\to\R$. In this way we give a geometric meaning to the hierarchy of integrable systems introduced in \cite{DK1}.

Another applications of our results to the Veronese webs include: a construction of the canonical connections for the Veronese webs (Theorem \ref{thm3}) and a local characterisation of the flat Veronese webs in terms of the torsion of the canonical connection (Corollary \ref{cor_webs}). Moreover, we give new, elementary proof of the so-called Zakharevich conjecture \cite{P} (Corollary \ref{cor_webs2}). All these results translate to bi-Hamiltonian structures via the Gelfand-Zakharevich reduction \cite{GZ}.


\section{Paraconformal structures and connections}\label{sec_paraconf}

Let $M$ be a manifold of dimension $k+1$. A paraconformal structure on $M$ is a vector bundle isomorphism
$$
TM\simeq \underbrace{S\odot S\odot\cdots\odot S}_k
$$
where $S$ is a rank-two vector bundle over $M$ and $\odot$ denotes the symmetric tensor product. It follows that any tangent space $T_xM$ is identified with the space of homogeneous polynomials of degree $k$ in two variables. The natural action of $GL(2,\R)$ on $S$ extends to the irreducible action on $TM$ and reduces the full frame bundle to a $GL(2,\R)$-bundle. Therefore the paraconformal structures are sometimes called $GL(2,\R)$-geometries. We refer to \cite{B,DT} for more detailed descriptions of the paraconformal structures.

A paraconformal structure defines the following cone 
$$
C(x)=\{v\odot\cdots\odot v\ |\ v\in S(x)\}\subset T_xM
$$
at each point $x\in M$ and it is an easy exercise to show that the field of cones $x\mapsto C(x)$ defines the paraconformal structure uniquely. If a basis $e_0,e_1$ in $S(x)$ is chosen then any $v\in S(x)$ can be written as $v=se_0+te_1$ and then
$$
C(x)=\{s^kV_0+s^{k-1}tV_1+\cdots+t^kV_k\ |\ (s,t)\in\R^2\}
$$
where $V_i=\binom{k}{i}e_0^{\odot k-i}\odot e_1^{\odot i}$. We shall denote
$$
V(s,t)=s^kV_0+s^{k-1}tV_1+\cdots+t^kV_k
$$
and refer to the vectors as null vectors. The cone $C(x)$ defines a rational normal curve $(s:t)\mapsto \R V(s,t)$ of degree $k$ in the projective space $P(T_xM)$. Sometimes, for convenience, we will use an affine parameter $t=(1:t)$ and denote $V(t)=V_0+tV_1+\cdots+t^kV_k$. Derivatives of $V(t)$ with respect to $t$ will be denoted $V'(t)$, $V''(t)$ etc. Let us stress that the parameter $t$ depends on the choice of a basis in $S$. However we shall use it in order to have a convenient description of the paraconformal structure.

We will consider connections $\nabla$ on $M$ which are compatible with the projective structure in a sense that the parallel transport preserves the null vectors i.e. it preserves the field of cones $x\mapsto C(x)$. Precisely we have
\begin{definition}
A connection $\nabla$ is called \emph{paraconformal} for a given paraconformal structure $x\mapsto C(x)$ on a manifold $M$ if 
$$
\nabla_YV(t)\in\spn\{V(t),V'(t)\}
$$
for any $t\in\R$ and any vector field $Y$ on $M$.
\end{definition}

From the point of view of $GL(2,\R)$-structures, the connections satisfying the above condition are in a one to one correspondence with the principal $GL(2,\R)$-connections.

We are interested in the properties of the geodesics of $\nabla$. Therefore, at least at this point, we will not impose any additional assumptions on the torsion of a connection. Let us only remark here that in low dimensions ($k=1,2$) there are plenty of torsion-free connections adapted to a paraconformal structure. On the other hand, already in the case $k=3$ any connection adapted to a \emph{generic} paraconformal structure has a torsion but in the most interesting case related to ODEs there is a unique torsion-free connection \cite{B}. 

Let us fix a point $x\in M$. We define the following 1-parameter family of $i$-dimensional subspaces of $T_xM$ for any number $i\in\{1,\ldots,k\}$
\begin{equation}\label{eq_Vi}
\VV_i(t)(x)=\spn\{V(t)(x),V'(t)(x),V''(t)(x),\ldots,V^{(i-1)}(x)\}.
\end{equation}
The family $\{\VV_i(t)(x)\ |\ t\in\R,\ x\in M\}$, for any $i$, is canonically defined by the paraconformal structure itself, although the choice of the parameter $t$ is not canonical. The hyperplanes $\VV_k(t)(x)$ will be referred to as $\alpha$-planes of the structure. In what follows we will consider paraconformal structures with an adapted connection such that the $\alpha$-planes are tangent to totally geodesic submanifolds of $M$. Two problems arise. First of all, the subspaces $\VV_k$ have to be tangent to submanifolds of $M$ and this issue does not depend on $\nabla$. The second problem is how to make a submanifold totally geodesic with respect to some connection. We will show that there are obstructions for the existence of such connections. In terms of ODEs the obstructions are expressed by new point invariants.

In order to guarantee the integrability of $\VV_k$ we shall consider the following notions
\begin{definition}
A co-dimension one submanifold $N\subset M$ is called an \emph{$\alpha$-submanifold} of a paraconformal structure on $M$ if all tangent spaces $T_xN$, $x\in N$, are $\alpha$-planes of the paraconformal structure.
A paraconformal structure is \emph{$\alpha$-integrable} if any $\alpha$-plane is tangent to some $\alpha$-submanifold of $M$. 
\end{definition}
In the next section we shall prove that all $\alpha$-integrable paraconformal structures can be defined in terms of special ODEs.

\section{ODEs and paraconformal structures}\label{sec_ODEinv}

Paraconformal structures can be constructed out of ODEs. We will consider ODEs in the following form
$$
x^{(k+1)}=F(t,x,x',\ldots,x^{(k)}).\eqno{(F)}
$$
The following theorem is a compilation of results of Chern \cite{Ch}, Bryant \cite{B}, Dunajski and Tod \cite{DT} (see also \cite{FKN,GN,N}).
\begin{theorem}\label{thm1}
If the W\"unschmann invariants of $(F)$ vanish then the solution space of $(F)$ possesses a canonical paraconformal structure.
\end{theorem}

To explain the meaning of the theorem and give an insight into its proof we recall that the geometry of an ODE of order $k+1$ is described on a manifold of $k$-jets, denoted $J^k(\R,\R)$. There is a canonical projection $\pi$ from $J^k(\R,\R)$ to the solution space $M_F$ with one-dimensional fibre tangent to the total derivative vector field
$$
X_F=\partial_t+x_1\partial_0+x_2\partial_1+\cdots+F\partial_k,
$$
where $t,x_0,x_1,\ldots,x_k$ are standard coordinates on the space of jets and $\partial_i=\frac{\partial}{\partial x_i}$. It follows that $M_F=J^k(\R,\R)/X_F$. The term \emph{canonical} in the theorem means that the null vectors of the paraconformal structure are tangent to $\pi_*\partial_{k}$. An equation of order $k+1$ has $k-1$ W\"unschmann invariants (or strictly speaking relative invariants). In particular there are no invariants for equations of order 2. There is one invariant in order three. This invariant was originally defined by W\"unschmann \cite{Wu} and later used by Chern \cite{Ch}. The two invariants in order four were introduced by Bryant \cite{B}. The general case was treated by Dunajski and Tod \cite{DT}. We use the name W\"unschmann invariants in all cases for convenience and because all invariants have similar nature. Actually, in the linear case, all of them were defined already by Wilczynski \cite{W}. Doubrov \cite{Db,Db2} generalised the Wilczynski invariants to non-linear case by computing Wilczynski invariants for the linearised equation. It appears that this procedure also gives the W\"unschmann invariants, c.f. \cite{DT}. In what follows we will sometimes say that the W\"unschmann condition (or Bryant condition in the case of order 4) holds if all W\"unschmann invariants (or equivalently the generalised Wilczynski invariants of Doubrov) vanish.

In the present paper the following approach to the W\"unschmann invariants will be useful. One looks for sections of $\VV=\spn\{\partial_k\}$ and $\XX_F=\spn\{X_F\}$, which are necessarily of the form $g\partial_k$ and $fX_F$ for some functions $f$ and $g$, and imposes the condition
\begin{equation}\label{eq_a}
\ad_{fX_F}^{k+1}g\partial_k=0\mod g\partial_k,\ad_{fX_F}g\partial_k,\ad_{fX_F}^2g\partial_k,\ldots,\ad_{fX_F}^{k-2}g\partial_k,X_F,
\end{equation}
where $\ad_XY=[X,Y]$ is the Lie bracket of vector fields and $\ad_X^{i+1}Y=[X,\ad_X^iY]$. One can prove that such $f$ and $g$ always exist (see Proposition 4.1 \cite{K1}) and then 
$$
\ad_{fX_F}^{k+1}g\partial_k=L_0g\partial_k+L_1\ad_{fX_F}g\partial_k+\ldots,+L_{k-2}\ad_{fX_F}^{k-2}g\partial_k\mod X_F.
$$
for some coefficients $L_i$. Then, there exist rational numbers $c_{ij}\in\Q$ such that the W\"unschmann invariants are given by the formulae
$$
W_i=L_i+\sum_{j>i}c_{ij}(fX_F)^{j-i}(L_j).
$$
In particular the vanishing of all $W_i$ is equivalent to the vanishing of all $L_i$. The construction described above is a non-linear version of a construction of the Halphen normal form and reproduces the Wilczynski invariants for linear equations \cite{W}. Moreover, if the W\"unschmann invariants vanish then
\begin{equation}\label{eq_wun}
\ad_{fX_F}^{k+1}g\partial_k=0\mod \XX_F
\end{equation}
and it follows that $g\partial_k$ depends polynomially on a parameter on integral curves of $\XX_F$. It implies that the projection of $g\partial_k$ to the solution space $J^k(\R,\R)/X_F$ defines a field of rational normal curves in $P(TM_F)$. This completes a sketch of the proof of Theorem \ref{thm1}. The theorem can be strengthen to the following theorem
\begin{theorem}\label{thm1b}
If the W\"unschmann invariants of $(F)$ vanish then the corresponding paraconformal structure on the solution space is $\alpha$-integrable. Conversely, all $\alpha$-integrable paraconformal structures can be locally obtained in this way. 
\end{theorem}
\begin{proof}
Let
\begin{equation}\label{eq_Dk}
\DD_k=\spn\{\partial_k,\partial_{k-1},\ldots,\partial_{1}\}
\end{equation}
be an integrable corank 2 distribution on the space of jets $J^k(\R,\R)$. Note that $\DD_k$ is tangent to the fibres of the projection $J^k(\R,\R)\to J^0(\R,\R)$. Equivalently
$$
\DD_k=\spn\{g\partial_k,\ad_{fX_F}g\partial_k,\ad_{fX_F}^2g\partial_k,\ldots,\ad_{fX_F}^{k-1}g\partial_k \}\mod X_F,
$$
for arbitrary nowhere vanishing functions $f$ and $g$. In particular one can take $f$ and $g$ as in \eqref{eq_wun}. Then, it follows from the construction presented above that the projection of $\DD_k$ to the solution space $J^k(\R,\R)/X_F$ gives a 2-parameter family of $\alpha$-submanifolds and any $\alpha$-plane of the paraconformal structure is tangent to some submanifold from this family. Indeed, the null directions of the paraconformal structure are defined by $\pi_*g\partial_k$ and the manifold $M_F$ is the quotient space $J^k(\R,\R)/X_k$. Thus, $\pi_*\DD_k$ are of the form \eqref{eq_Vi}, where as the parameter $t$ one takes a parametrisation of integral lines of $fX_F$. Therefore, the structure is $\alpha$-integrable.

In order to prove the second part we shall use \cite{K1}. First of all we associate to a paraconformal structure a pair of  distributions $(\XX,\VV)$ on the fibre bundle $P(C)$ over $M$, where $P(C)$ is the projectivisation of the null cone of the paraconformal structure. We define $\XX$ as the distribution tangent to the fibres of $P(C)$. Thus $\rk\,\XX=1$. $\VV$ is the tautological distribution on $P(C)\subset P(TM)$. Then $\rk\,\VV=2$ and $\XX\subset\VV$. According to \cite{K1} it is sufficient to prove that the pair $(\XX,\VV)$ is of equation type, i.e. $\VV$ is locally diffeomorphic to the Cartan distribution on $J^k(\R,\R)$. We first note that due to the fact that $C(x)$ is a rational normal curve of degree $k$ the pair $(\XX,\VV)$ defines the following flag
\begin{equation}\label{flag}
\VV\subset\ad_\XX\VV\subset\ldots\subset\ad^{k-1}_\XX\VV\subset\ad^k_\XX\VV=TP(C)
\end{equation}
where $\rk\,\ad^i_\XX\VV=i+2$ and for two distributions $\mathcal{Y}_1$ and $\mathcal{Y}_2$ we define
$$
[\mathcal{Y}_1,\mathcal{Y}_2]=\spn\{[Y_1,Y_2]\ |\ Y_1\in\Gamma(\mathcal{Y}_1),\ Y_2\in\Gamma(\mathcal{Y}_2)\}
$$
and then inductively $\ad^{i+1}_{\mathcal{Y}_1}\mathcal{Y}_2=[\mathcal{Y}_1,\ad^i_{\mathcal{Y}_1}\mathcal{Y}_2]$
(c.f. \cite{K1}). In terms of \eqref{eq_Vi}
$$
(\ad^i_\XX\VV)(x,t)=\pi_*^{-1}(\VV_{i+1}(x)(t))
$$
where $\pi\colon P(C)\to M$ is the projection, $x\in M$ and $t\in P(C)(x)$ is an affine coordinate. Further, since $P(C)$ parametrises all $\alpha$-planes, the $\alpha$-integrability of the structure gives a foliation of $P(C)$ of co-dimension 2. The tangent bundle of the foliation is an integrable sub-distribution $\FF\subset \ad^{k-1}_\XX\VV$. We set $\FF_{k-1}=\FF$ and define
$$
\FF_{i-1}=\spn\{Y\in\Gamma(\FF_i)\ |\ [X,Y]\in\Gamma(\ad_\XX^i\VV),\ X\in\Gamma(\XX)\},\qquad i=k-1,\ldots,1.
$$
Then $\FF_{i-1}$ is a sub-distribution of $\ad_\XX^{i-1}\VV$ of co-rank 1, because
$$
\ad_\XX^{i-1}\VV=\spn\{Y\in\Gamma(\ad^i_\XX\VV)\ |\ [X,Y]\in\Gamma(\ad_\XX^i\VV),\ X\in\Gamma(\XX)\}
$$
and $\FF$ is of co-rank 1 in $\ad_\XX^{k-1}\VV$.
We shall prove that $\FF_{i-1}$ is integrable provided that $\FF_i$ is integrable. Let $Y_1, Y_2\in\Gamma(\FF_{i-1})$. Then $[Y_1,Y_2]$ is a section of $\FF_i$ because $\FF_{i-1}\subset\FF_i$. Moreover, since $Y_j\in\Gamma(\ad_\XX^{i-1}\VV)$, $j=1,2$, we have $[X,Y_j]\in\Gamma(\ad^i_\XX\VV)$. Hence, $[X,Y_j]=Z_j+f_jX$, $j=1,2$, where $Z_j\in\Gamma(\FF_i)$ and $f_j$ is a function. The Jacobi identity reads
$$
[X,[Y_1,Y_2]]=[Z_1, Y_2]+[Y_1, Z_2]+f_1[X,Y_2]-f_2[X,Y_1]\mod \XX.
$$
The right hand side is a section of $\ad_\XX^i\VV$. Thus, $[Y_1,Y_2]$ is a section of $\FF_{i-1}$ and consequently $\FF_{i-1}$ is integrable.

We have proved that all $\ad^i_\XX\VV$ contain integrable, co-rank one sub-distributions $\FF_i$ such that $\ad^i_\XX\VV=\FF_i\oplus\XX$. It follows that $[\ad^i_\XX\VV,\ad^i_\XX\VV]=[\XX,\ad_\XX^i\VV]=\ad^{i+1}_\XX\VV$. Thus \eqref{flag} is a regular Goursat flag, c.f. \cite{A,MZ}, because of $\rk\,\ad^i_\XX\VV=i+2$. This completes the proof.
\end{proof}

\begin{remark}
Theorem \ref{thm1b} can be considered as a generalisation to higher dimensions of the 3-dimensional case \cite{Ch,FKN}. Indeed, in dimension 3 all paraconformal structures can be obtained from ODEs \cite{FKN} and in this dimension all paraconformal structures are $\alpha$-integrable. In higher dimensions one needs to assume that a paraconformal structure is $\alpha$-integrable in order to be defined by an ODE.
\end{remark}

The construction of the W\"unschmann invariants presented above can be split into two steps. One look first for a function $g$ and then for $f$. The first step already gives interesting results. Namely, \eqref{eq_a} can be weakened to
\begin{equation}\label{eq_b}
\ad_{X_F}^{k+1}g\partial_k=0\mod g\partial_k,\ad_{X_F}g\partial_k,\ad_{X_F}^2g\partial_k,\ldots,\ad_{X_F}^{k-1}g\partial_k
\end{equation}
and such $g$ always exists. This gives $k$ coefficients $K_0,K_1,\ldots,K_{k-1}$ defined by the formula
$$
\ad_{X_F}^{k+1}g\partial_k= -K_0g\partial_k+K_1\ad_{X_F}g\partial_k-K_2\ad_{X_F}^2g\partial_k+\ldots+(-1)^{k-1}K_{k-1}\ad_{X_F}^{k-1}g\partial_k
$$
(we add the minus signs for convenience). The coefficients, called curvatures in \cite{J,JK}, have well defined geometric meaning. They are invariant with respect to contact transformations that do not change the independent variable $t$. The class of transformations was called time-preserving contact transformations (or contact-affine transformations) in \cite{JK}. The class gives a natural framework in the context of control mechanical systems \cite{J,JK}, Finsler geometry (in this case $K_0$ is the flag curvature) and webs \cite{K2}. In the present paper we will use the invariants $K_i$ to write down more complicated objects in a simple form (compare \cite{G} in the case of second order). The curvatures $K_i$ can be explicitly computed in terms of the original equation $(F)$ using \cite[Proposition 2.9]{JK}. We will provide the formulae in the case of equations of order 2, 3 and 4 in Appendix \ref{ap_formulae}.

A function $g$ defined by \eqref{eq_b} is a non-trivial solution to
\begin{equation}\label{eq_g}
X_F(g)=\frac{g}{k+1}\partial_kF.
\end{equation}
We will use the notation
$$
V=g\partial_k.
$$

In the subsequent sections we will extensively use the Lie derivative $\mathcal{L}_{X_F}$ acting on different objects. If not mentioned otherwise the terms ``derivative'' or ``differentiation'' will refer to $\mathcal{L}_{X_F}$. Moreover, we will denote differentiations by adding primes to the objects. In particular we will have
$$
V'=\ad_{X_F}V,\quad V''=\ad_{X_F}^2V,\quad\ldots\quad, V^{(j)}=\ad_{X_F}^jV
$$
for the vector field $V$ or
$$
K_i'=X_F(K_i),\quad K_i''=X_F^2(K_i),\quad\ldots\quad, K_i^{(j)}=X_F^j(K_i)
$$
for the curvatures $K_i$.

\section{ODEs and connections}\label{sec_ODEcon}

We assume that an ODE $(F)$ defines a paraconformal structure on $M_F$ via Theorem~\ref{thm1}, i.e. all W\"unschmann invariants vanish. Let us introduce on the space of jets $J^k(\R,\R)$ the following integrable distributions (generalising \eqref{eq_Dk})
\begin{equation}\label{eq_Di}
\DD_i=\spn\{\partial_k,\partial_{k-1},\ldots,\partial_{k-i+1}\}=\spn\{V,V',\ldots,V^{(i-1)}\}
\end{equation}
 which are tangent to the fibres of the natural projections $J^k(\R,\R)\to J^{k-i}(\R,\R)$, for $i=1,\ldots,k$. The distributions can be projected to the solution space $M_F$. The projections give exactly the subspaces $\VV_i\subset TM_F$ defined before by formula \eqref{eq_Vi} for a paraconformal structure. Therefore, one can ask if the projection of leaves of $\DD_k$ defines a two-parameter family of totally geodesic hypersurfaces in $M$. If yes, then we shall consider ODEs up to point transformations, i.e. transformations of variables $t$ and $x$ only, because in terms of jets we get precisely contact transformations preserving $\DD_k$. It follows that there is a double fibration picture
\begin{equation}\label{diag_fibr}
M_F\longleftarrow J^k(\R,\R)\longrightarrow B
\end{equation}
where $B=J^0(\R,\R)$ is the space where $(F)$ is defined and $M_F$ is the solution space as before.
\begin{definition}
A class of point equivalent equations admits a \emph{totally geodesic paraconformal connection} if the projections of the integral manifolds of $\DD_k$ to the solution space $M_F$ are totally geodesic submanifolds with respect to a paraconformal connection on $M_F$.
\end{definition}
 
In order to construct a paraconformal connection on $M_F$ we will construct a connection on $J^k(\R,\R)$ which is ``invariant'' along $\XX_F$ and then we will project it to $M_F$. Precisely, if $\nabla$ is a connection on $J^k(\R,\R)$ then we would like to define a connection $\tilde \nabla$ on $M_F$ by the formula
\begin{equation}\label{proj_nabla}
\tilde\nabla_{Y_1}Y_2=\pi_*\nabla_{\pi_*^{-1}Y_1}\pi_*^{-1}Y_2.
\end{equation}
The definition is correct only for special $\nabla$. There are two difficulties. Firstly, the lifts $\pi_*^{-1}Y_i$ are given modulo $\XX_F$ only. Secondly, $\nabla$ may depends on a point in the fibre of $\pi$.  To overcome the difficulties we need several additional conditions.

\begin{lemma}\label{lemma0}
A connection $\nabla$ on $J^k(\R,\R)$ defines a connection $\tilde \nabla$ on $M_F$ via \eqref{proj_nabla} if and only if
\begin{enumerate}
\item $\nabla_YX=0\mod \XX_F$,
\item $\nabla_{X}Y=[X,Y]\mod \XX_F$,
\item $\mathcal{L}_X\nabla Y=\nabla [X,Y]\mod \Omega^1(J^k(\R,\R))\otimes\XX_F$,
\end{enumerate}
where $X$ is an arbitrary section of $\XX_F$ and $Y$ is an arbitrary vector field on $J^k(\R,\R)$.
\end{lemma}
\begin{proof}
The first two conditions are equivalent to the fact that $\nabla_{\pi_*^{-1}Y_1}\pi_*^{-1}Y_2 \mod\XX_F$ does not depend on the lift of $Y_1$ or $Y_2$ to $J^k(\R,\R)$. The third condition (together with the first one) is equivalent to the fact that $\mathcal{L}_X\nabla Y=0\mod\XX_F$ for $Y$ being a lift of a vector field on $M_F$. It means that $\pi_*\nabla Y$ is well defined independently on the point in the fibre of $\pi$, hence defines a connection on $M_F$.
\end{proof}

If we assume that equation $(F)$ satisfies the W\"unschmann condition and a connection $\nabla$ on $J^1(\R,\R)$ satisfies the three conditions given in Lemma \ref{lemma0}, then the connection $\tilde \nabla$ on $M_F$ will be compatible with the paraconformal structure on $M_F$ defined by $(F)$ if and only if 
\begin{equation}\label{cond_nabla}
\nabla V=\alpha V+\beta V',
\end{equation}
for some two one-forms $\alpha$ and $\beta$ on $J^k(\R,\R)$.

\begin{lemma}\label{lemma1}
The one-forms $\alpha$ and $\beta$ satisfy the following system of differential equations
\begin{eqnarray}
&&\alpha'+k\beta''=0,\nonumber\\
&&\left(\binom{k+1}{j-1}-\frac{k}{2}\binom{k+1}{j}\right)\beta^{(k-j+2)} +(-1)^{j+1}K_j'\beta+(-1)^{j+1}(k-j+1)K_j\beta'\nonumber \\ 
&&\qquad=(-1)^{j+1}dK_j +\sum_{l=j+1}^{k-1}(-1)^{l+1} \left(\binom{l}{j-1}-\frac{k}{2}\binom{l}{j}\right)K_l\beta^{(l-j+1)}\label{system}
\end{eqnarray} 
for $j=0,\ldots,k-1$.
\end{lemma}
\begin{proof}
A consecutive application of Lemma \ref{lemma0} gives $\mathcal{L}^i_{X_F}\nabla V=\nabla V^{(i)} \mod \XX_F$. This written in terms of $\alpha$ and $\beta$ reads
\begin{equation}\label{eq_nabla}
\nabla V^{(i)}=\sum_{j=0}^{i+1}\left(\binom{i}{j}\alpha^{(i-j)}+ \binom{i}{j-1}\beta^{(i-j+1)}\right)V^{(j)}.
\end{equation}
The formula is valid for all $i$. For $i=1,\ldots,k$ it defines the connection uniquely (note that for $i=k$ the formula involves $K_j$'s via the last term $V^{(k+1)}=\sum_{j=0}^{k-1}(-1)^{j+1}K_jV^{(j)}$) and for $i=k+1$ it gives a set of conditions that should be satisfied by $\alpha$ and $\beta$. The conditions are as follows
\begin{eqnarray*}
&&\binom{k+1}{j}\alpha^{(k-j+1)}+\binom{k+1}{j-1}\beta^{(k-j+2)} +(-1)^{j+1}K_j'\beta+(-1)^{j+1}(k-j+1)K_j\beta' \\ 
&&\qquad=(-1)^{j+1}dK_j +\sum_{l=j+1}^{k-1}(-1)^{l+1} \left(\binom{l}{j}\alpha^{(l-j)}+ \binom{l}{j-1}\beta^{(l-j+1)}\right)K_l
\end{eqnarray*}
for $j=0,\ldots,k$. In particular, for $j=k$ we get
$$
2\alpha'=-k\beta''
$$
and using it we can eliminate derivatives of $\alpha$ from the remaining equations and obtain \eqref{system} as a result.
\end{proof}

We get the following result
\begin{theorem}\label{thm2}
An ODE of order $k+1$ with the vanishing W\"unschmann invariants admits a totally geodesic paraconformal connection if and only if there exists a one-form $\beta$ on $J^k(\R,\R)$ satisfying
\begin{equation}\label{eq_gen}
-\frac{1}{2}\binom{k+2}{3}\beta'''+(-1)^kK_{k-1}'\beta+2(-1)^kK_{k-1}\beta'=(-1)^kdK_{k-1}
\end{equation}
and
\begin{equation}\label{cond_gen}
\beta(V)=\beta(V')=\cdots=\beta(V^{(k-1)})=0.
\end{equation}
\end{theorem}
\begin{proof}
Assume first that an ODE admits a totally geodesic paraconformal connection. Then by Lemma \ref{lemma1} it satisfies System \eqref{system}. In particular, for $j=k-1$ one gets \eqref{eq_gen}. Moreover one should have
$$
\nabla_{V^{(j)}}V^{(k-1)}\in D_k\mod \XX_F
$$
for all $j\leq k-1$. But, it follows from \eqref{eq_nabla} that the coefficient of $\nabla_{V^{(j)}}V^{(k-1)}$ next to $V^{(k)}$ is exactly the one form $\beta$ evaluated on $V^{(j)}$. Therefore $\beta(V)=\beta(V')=\cdots=\beta(V^{(k-1)})=0$.

In order to prove the theorem in the opposite direction it is sufficient to show that if \eqref{eq_gen} has a solution $\beta$ and the W\"unschmann invariants vanish then $\beta$ solves also all other equations from the System \eqref{system}. But in Lemma \ref{lemma0} one can use an arbitrary section of $\XX_F$ instead of $X_F$. It is convenient to make all computations using a multiple of $X_F$ by a function $f$ as in \eqref{eq_wun}. Such a function $f$ exists due to the W\"unschmann condition. If \eqref{eq_wun} is satisfied then all $K_i$ in \eqref{system} are zero and the System \eqref{system} takes the form $\beta^{(k-j+2)}=0$, $j=0,\ldots,k-1$. The system clearly has a solution.
\end{proof}

\begin{remark}
A reasoning similar to the proof of Theorem \ref{thm2} implies that the projections to $M_F$ of the integral manifolds of $\DD_i$ are totally geodesic for a paraconformal connection if $\beta(V)=\beta(V')=\cdots=\beta(V^{(i-1)})=0$. In particular, if projections of the integral manifolds of $\DD_i$ are totally geodesic then also projections of the integral manifolds of $\DD_j$ for $j<i$ are totally geodesic. Let us use \eqref{eq_nabla} again and compute the following torsion coefficient
\begin{eqnarray*}
&&T(\nabla)(V^{(i)},V^{(i+1)})=\nabla_{V^{(i)}}V^{(i+1)}-\nabla_{V^{(i+1)}}V^{(i)}-[V^{(i)},V^{(i+1)}]\\
&&\qquad=\beta(V^{(i)})V^{(i+2)}\mod\DD_{i+2}.
\end{eqnarray*}
The last equality holds because $[V^{(i)},V^{(i+1)}]\in\DD_{i+2}$. The expression has sense for $i=0,\ldots,k-2$ and it follows that the condition $\beta(V)=\beta(V')=\cdots=\beta(V^{(i)})=0$ is expressed in terms of the torsion $T(\nabla)(V^{(i)},V^{(i+1)})$ for $i=0,\ldots,k-2$. However, the condition \eqref{cond_gen} for $i=k-1$ has a different nature.
\end{remark}

\begin{remark}
Instead of using in the proof of Theorem \ref{thm2} the vector field $fX_F$ satisfying \eqref{eq_wun} one can differentiate \eqref{eq_gen} sufficiently many times and subtract it from the remaining equations from \eqref{system} in such a way that the highest derivatives of $\beta$ are eliminated. Then one will recover the W\"unschmann condition as vanishing of coefficients next to the derivatives of $\beta$ of lower order. Conditions are given in terms of $K_i$'s. In particular in the case of an equation of order 3 we get
\begin{equation}\label{wun_ord3}
K_0+\frac{1}{2}K_1'=0
\end{equation}
and it can be checked that $W_0=K_0+\frac{1}{2}K_1'$ is really the W\"unschmann invariant ($K_0$ and $K_1$ are given explicitly below in Appendix \ref{ap_formulae}).

In the case of order 4 we get
\begin{eqnarray}
&&K_0+\frac{3}{10}K_1'-\frac{9}{100}K_2^2=0,\label{wun0_ord4}\\
&&K_1+K_2'=0.\label{wun1_ord4}
\end{eqnarray}
We have computed that the conditions coincide with the conditions in \cite[Theorem 1.3]{DT} and consequently with \cite{B}  (again, $K_0$, $K_1$ and $K_2$ are given explicitly below in Appendix \ref{ap_formulae}). Namely \eqref{wun1_ord4} is exactly the second condition in \cite{DT} and $W_1=K_1+K_2'$ is the W\"unschmann invariant. The first condition in \cite{DT} has the form
$$
K_0+K_1'+\frac{7}{10}K_2''-\frac{9}{100}K_2^2-\frac{1}{4}\partial_3F(K_1+K_2')=0
$$
which is \eqref{wun0_ord4} modulo \eqref{wun1_ord4} and the derivative of \eqref{wun1_ord4}.

In the general case the simplest W\"unschmann condition has the form
\begin{equation}\label{wun_gen}
K_{k-2}+\frac{k-1}{2}K_{k-1}'=0
\end{equation}
The other are more complicated, but we will not need them in the explicit form.
\end{remark}

\section{Twistor correspondence}\label{sec_twistor}
The condition \eqref{cond_gen} means that the one-form $\beta$ is a pullback of a one-form defined on the space $B=J^0(\R,\R)$. One can call $B$ the twistor space. Indeed, due to the double fibration \eqref{diag_fibr} a point in $B$ can be considered as a hypersurface in $M_F$ and a point in $M_F$ is represented by a curve in $B$ which is a solution to $(F)$. There is $(k+1)$-parameter family of such curves corresponding to different points in $M_F$. 

In the complex setting one can repeat the reasoning of \cite[Section 5]{H}. One considers a complex surface $B$ and a curve $\gamma\subset B$ with a normal bundle $N_\gamma\simeq O(k)$. Then $H^0(\gamma,N_\gamma)=\C^{k+1}$ and $H^1(\gamma,N_\gamma)=0$. Therefore by the Kodaira theorem one gets a $(k+1)$-dimensional complex manifold $M$ parametrising a family of curves in $B$ with self intersection number $k$. One can see a paraconformal structure in this picture, geodesics of an adapted connection and a set of totally geodesic surfaces. Indeed, if $\gamma$ is a curve in $B$  then due to $N_\gamma\simeq O(k)$ we get that for any collection of points $\{y_1,\ldots,y_k\}$, $y_i\in \gamma$, possibly with multiplicities, there is a one-parameter family of curves in $B$ which intersect $\gamma$ exactly at these points. The family of curves defines a geodesic in $M$ (the fact that such a definition gives geodesics of a connection can be proved exactly as in \cite{B} and follows from the fact $H^1(\gamma,N_\gamma)=0$). The null geodesics are defined by $y_i$'s such that $y_1=y_2=\cdots=y_k$. A totally geodesic hypersurface in $M$ corresponding to a point $y\in B$ is defined by all curves which pass through $y$. It follows automatically from the definition of the geodesics that such hypersurfaces are totally geodesic indeed.

One gets from the twistor construction not a unique connection but rather a set of unparameterised geodesics i.e.~a class of connections sharing the geodesics. The connections have in general non-vanishing torsion if we impose that they are adapted to the paraconformal structure (see \cite{MP} for a twistor construction leading to torsion-free structures). We shall call the structure a \emph{projective structure with a torsion}. Clearly, one can consider torsion-free connections shearing the geodesics, but then the compatibility with the paraconformal structure is lost. Anyway, we shall prove later that starting from dimension 4 (the classical dimensions 2 and 3 considered in \cite{H} are slightly different) the one form $\beta$ is unique and the set of projectively equivalent paraconformal connections depends on an arbitrary one-form $\alpha$ on $M$ as in the case of the torsion-free connections (Corollaries \ref{cor_ord4} and \ref{cor_ordgen}).

\section{Second order}\label{sec_ord2}
Let
$$
x''=F(t,x,x'')
$$
be a second order ODE. The tangent bundle to a two-dimensional manifold has a natural $GL(2,\R)$-structure. Hence, any second order equation defines a paraconformal structure on its solution space. However, the existence of a totally geodesic paraconformal connection is a more restrictive condition which is equivalent to the existence of a projective structure. A result due to Cartan \cite{C1} says that a class of point equivalent ODEs defines a projective structure on the solution space if and only if the Cartan invariant $C$ vanishes. In coordinates (see \cite{CS})
\begin{eqnarray*}
C&=&\partial_0^2F-\frac{1}{2}F\partial_0\partial_1^2F-\frac{1}{2}\partial_0F\partial_1^2F -\frac{2}{3}\partial_t\partial_0\partial_1F +\frac{1}{6}\partial_t^2\partial_1^2F+\\ &&\frac{1}{3}x_1\partial_t\partial_0\partial_1^2F+ \frac{1}{6}\partial_tF\partial_1^3F+ \frac{1}{3}F\partial_t\partial_1^3F- \frac{2}{3}x_1\partial_0^2\partial_1F +\frac{1}{6}x_1^2\partial_0^2\partial_1^2F+\\
&&\frac{1}{6}x_1\partial_0F\partial_1^3F+ \frac{1}{3}x_1F\partial_0\partial_1^3F+ \frac{2}{3}\partial_1F\partial_0\partial_1F- \frac{1}{6}\partial_1F\partial_t\partial_1^2F-\\
&&\frac{1}{6}x_1\partial_1F\partial_0\partial_1^2F+ \frac{1}{6}F^2\partial_1^4F.
\end{eqnarray*}
On the other hand Theorem \ref{thm2} specified to $k=1$ implies that the existence of a totally geodesic paraconformal connection is equivalent to the existence of a solution to
\begin{equation}\label{eq_ord2}
-\frac{1}{2}\beta'''-K_0'\beta-2K_0\beta'=-dK_0
\end{equation}
satisfying
\begin{equation}\label{cond_ord2}
\beta(V)=0.
\end{equation}
Thus, we reproduce Cartan's result in the following form 

\begin{theorem}\label{thm_ord2}
A class of point equivalent second order ODEs defines a projective structure on its solution space if and only if
\begin{equation}\label{cartan_ord2}
4V'(K_0)-V(K_0')=0.
\end{equation}
Additionally $C=4V'(K_0)-V(K_0')$.
\end{theorem}
\begin{proof}
We are looking for a common solution to \eqref{eq_ord2} and \eqref{cond_ord2}.
Let us denote
$$
\beta(V')=b.
$$
Taking into account that $V''=-K_0V$ and differentiating \eqref{cond_ord2} one finds
$$
\beta'(V)=-b,\quad \beta''(V)=-2b',\quad \beta'''(V)=-3b''+K_0b
$$
and
$$
\beta'(V')=b',\quad \beta''(V')=b''-K_0b,\quad \beta'''(V')=b'''-K_0'b-3K_0b'.
$$
Thus, evaluating \eqref{eq_ord2} on $V$ and $V'$ one gets
\begin{eqnarray*}
&&\frac{3}{2}b''+\frac{3}{2}K_0b=-V(K_0),\\
&&\frac{1}{2}b'''+\frac{1}{2}K_0'b+\frac{1}{2}K_0b'=V'(K_0).
\end{eqnarray*}
Differentiating the first equation and substituting to the second one one gets \eqref{cartan_ord2}. Besides one can check by direct computations using formulae in Appendix \ref{ap_formulae} that \eqref{cartan_ord2} coincides with $C$. 
\end{proof}

\section{Third order}\label{sec_ord3}
Let
$$
x'''=F(t,x,x',x'')
$$
be a third order ODE. Its solution space is a three dimensional manifold. A paraconformal structure on a three dimensional manifold is a conformal metric $[\mathbf{g}]$ of Lorentzian signature. Moreover, a torsion-free connection adapted to a paraconformal structure is a Weyl connection $\nabla$ for $[\mathbf{g}]$. We recall that if a representative $\mathbf{g}\in[\mathbf{g}]$ is chosen then a Weyl connection $\nabla$ is uniquely defined by a one-form $\varphi$ such that
$$
\nabla\mathbf{g}=\varphi\mathbf{g}.
$$
According to Cartan \cite{C2}, a Weyl connection $\nabla$ is totally geodesic in our sense if and only if the Einstein equation is satisfied
$$
Ric(\nabla)_{sym}=\frac{1}{3}R_\mathbf{g}(\nabla)\mathbf{g}
$$
where $Ric(\nabla)_{sym}$ is the symmetric part of the Ricci curvature of $\nabla$ and $R_\mathbf{g}(\nabla)$ is the scalar curvature with respect to $\mathbf{g}$. The pair $([\mathbf{g}],\nabla)$ is called an Einstein-Weyl structure in this case. Cartan also proved that there is a one to one correspondence between Einstein-Weyl structures and third order ODEs for which the W\"unschmann $W_0$ and Cartan $C$ invariants vanish (see \cite{C3,T}). In coordinates
\begin{eqnarray*}
W_0&=&\partial_0F -\frac{1}{2}X_F(\partial_1F) +\frac{1}{3}\partial_1F\partial_2F+\frac{1}{6}X_F^2(\partial_2F)
 -\frac{1}{3}X_F(\partial_2F)\partial_2F+\frac{2}{27}(\partial_2F)^3,\\
C&=&X_F^2(\partial_2^2F)-X_F(\partial_1\partial_2F)+\partial_0\partial_2F.
\end{eqnarray*}

On the other hand Theorem \ref{thm2} implies that if $W_0=0$ then the existence of a totally geodesic paraconformal connection is equivalent to the existence of a solution to
\begin{equation}\label{eq_ord3}
-2\beta'''+K_1'\beta+2K_1\beta'=dK_1
\end{equation}
satisfying
\begin{equation}\label{cond_ord3}
\beta(V)=\beta(V')=0.
\end{equation}
We reproduce Cartan's result in the following way

\begin{theorem}\label{thm_ord3}
A class of point equivalent third order ODEs defines an Einstein-Weyl structure on its solution space if and only if $W_0=0$ and
\begin{equation}\label{cartan_ord3}
2V'(K_1)+V(K_1')=0
\end{equation}
Additionally $C=-\frac{3}{2}(V'(K_1)-V(K_0))$ and under the W\"unschmann condition $2V'(K_1)+V(K_1')=2(V'(K_1)-V(K_0))=-\frac{3}{4}C$.
\end{theorem}
\begin{proof}
The formula for $C$ in terms of $K_0$ and $K_1$ can be verified by computations using the appropriate formulae given in Appendix \ref{ap_formulae}. The invariant meaning of this expression follows from our proof. We are looking for a common solution to \eqref{eq_ord3} and \eqref{cond_ord3}. Let us denote
$$
\beta(V'')=b.
$$
Taking into account that $V'''=-K_0V+K_1V'$ and differentiating \eqref{cond_ord3} one finds
$$
\beta'(V)=0,\quad \beta''(V)=b,\quad \beta'''(V)=3b',
$$
$$
\beta'(V')=-b,\quad \beta''(V')=-2b',\quad \beta'''(V')=-3b''-K_1b,
$$
and
$$
\beta'(V'')=b',\quad \beta''(V'')=b''+K_1b,\quad \beta'''(V'')=b'''+K_1'b+3K_1b'+K_0b.
$$
Thus, evaluating \eqref{eq_ord3} on $V$, $V'$ and $V''$ one gets
\begin{eqnarray*}
&&6b'=-V(K_1),\\
&&6b''=V'(K_1),\\
&&2b'''+K_1'b+4K_1b'+2K_0b=-V''(K_1).
\end{eqnarray*}
The last equation reduces to $-2b'''-4K_1b'=V''(K_1)$ due to the W\"unschmann condition $W_0=0$ which is equivalent to $K_1'=-2K_0$. Then, differentiating the first equation one gets that a common solution $b$ exists if and only if
$V'(K_1)+\frac{1}{2}V(K_1')=0$ and $K_1V(K_1)=2V''(K_1)+\frac{1}{2}V_1(K_1')$. We get \eqref{cartan_ord3} and using the W\"unschmann condition again
\begin{equation}\label{cartan_ord3'}
K_1V(K_1)=2V''(K_1)-V_1(K_0).
\end{equation}
Now, the theorem follows from the following
\begin{lemma}
If \eqref{cartan_ord3} holds and the W\"unschmann invariant vanishes for a third order ODE then also \eqref{cartan_ord3'} holds.
\end{lemma}
\begin{proof}
We can write $[V,V']=AV+BV'$ for some functions $A$ and $B$. Taking the Lie brackets with $X_F$ and using the fact that $V'''=-K_0V+K_1V'$ we get formulae for $[V,V'']$, $[V',V'']$ in terms of $A$, $B$ and their derivatives. Namely
$[V,V'']=A'V+(A+B')V'+BV''$ and $[V',V'']=(A''+V(K_0)-BK_0-AK_1)V+(2A'+B''-V(K_1))V'+(A+2B')V''$. One more Lie bracket and the Jacobi identity gives the following three equations
\begin{eqnarray*}
&&3A'+3B''-V(K_1)=0,\\
&&3A''+B'''-X_FV(K_1)+V(K_0)-2BK_0+2B'K_1-V'(K_1)=0,\\
&&A'''-3B'K_0-BK_0'-A'K_1+XV(K_0)+V'(K_0)=0.
\end{eqnarray*}
Differentiating the first equation and substituting to the second and third one we can eliminate $A$ and its derivatives. But due to the W\"unschmann condition we can also eliminate $B$ and its derivatives and get one relation
$$
2V'(K_0)+\frac{2}{3}X_F^2V(K_1)+2X_FV(K_0)-\frac{2}{3}K_1V(K_1).
$$
The subsequent use of the W\"unschmann condition and the relation $\ad_{X_F}^iV=V^{(i)}$ reads
$$
K_1V(K_1)=2V'(K_0)+V''(K_1)+V(K_0').
$$
On the other hand, differentiating \eqref{cartan_ord3} we get
$$
0=V''(K_1)-3V'(K_0)-V(K_0')
$$
and adding the last two equations we finally get \eqref{cartan_ord3'}.
\end{proof}
\end{proof}

\section{Fourth order}\label{sec_ord4}
Let
$$
x^{(4)}=F(t,x,x',x'',x''')
$$
be a fourth order ODE. This is the case considered by Bryant \cite{B}. However, the torsion free connections of \cite{B} are not, in general, totally geodesic in the sense of the present paper. Bryant proved in \cite[Theorem 4.1]{B} that a paraconformal structure possesses a torsion-free connection (unique) if and only if every null plane, i.e.\ every subspace $\VV_2(s:t)(x)\subset T_xM$ for $x\in M$ and $(s:t)\in \R P^1$ in the notation of Section \ref{sec_paraconf}, is tangent to a totally-geodesic surface in $M$. We argued in Section \ref{sec_ODEcon} (the first remark following Theorem \ref{thm2}) that this condition is really expressed in terms of the torsion. Actually, according to Bryant, this condition is also equivalent to the fact that a paraconformal structure is defined by an equation and it can be expressed as vanishing of a polynomial of degree 7 in $(s:t)$ (compare \cite{K1}). The corresponding ODE satisfies the Bryant-W\"unschmann condition and the geometry is related to the contact geometry of ODEs.

On contrary, Theorem \ref{thm2} describes paraconformal structures satisfying more restrictive conditions related to point geometry of ODEs. Namely, we assume that any subspace $\VV_3(s:t)(x)\subset T_xM$ for $x\in M$ and $(s:t)\in \R P^1$ is tangent to a totally-geodesic submanifold of $M$. We get that it happens if and only if the Bryant-W\"unschmann condition holds and additionally there is a solution to
\begin{equation}\label{eq_ord4}
5\beta'''+K_2'\beta+2K_2\beta'=dK_2
\end{equation}
satisfying
\begin{equation}\label{cond_ord4}
\beta(V)=\beta(V')=\beta(V'')=0.
\end{equation}
Of course, a torsion-free connection also exists in this case, but this connection is not necessarily totally geodesic in our sense as we will see in Section \ref{sec_Veronese}. We get the following new result

\begin{theorem}\label{thm_ord4}
A class of point equivalent fourth order ODEs admits a totally geodesic paraconformal connection if and only if the Bryant-W\"unschmann condition holds and
\begin{equation}\label{cartan_ord4}
4V'(K_2)+3V(K_2')=0.
\end{equation}
Additionally under the Bryant-W\"unschmann condition $4V'(K_2)+3V(K_2')=4V'(K_2)-3V(K_1)$.
\end{theorem}
\begin{proof}
The proof is similar to the proofs of Theorems \ref{thm_ord2} and \ref{thm_ord3}. Let us denote
$$
\beta(V''')=b.
$$
Taking into account that $V^{(4)}=-K_0V+K_1V'-K_2V''$ and differentiating \eqref{cond_ord4} one finds
$$
\beta'(V)=0,\quad \beta''(V)=0,\quad \beta'''(V)=-b,
$$
$$
\beta'(V')=0,\quad \beta''(V')=b,\quad \beta'''(V')=3b',
$$
$$
\beta'(V'')=-b,\quad \beta''(V'')=-2b',\quad \beta'''(V'')=-3b''+K_2b.
$$
and
$$
\beta'(V''')=b',\quad \beta''(V''')=b''-K_2b,\quad \beta'''(V''')=b'''-K_2'b-3K_2b'-K_1b.
$$
Thus, evaluating \eqref{eq_ord4} on $V$, $V'$, $V''$ and $V'''$ one gets
\begin{eqnarray*}
&&5b=-V(K_2),\\
&&15b'=V'(K_2),\\
&&15b''-3K_2b=-V''(K_2),\\
&&5b'''-4K_2'b-13K_2b'-5K_1b=V'''(K_2).
\end{eqnarray*}
If we differentiate the first equation and substitute it to the second one we get the condition \eqref{cartan_ord4}. From the last two equations we get two additional conditions of higher order. However, as in the case of order three, they are consequences of \eqref{cartan_ord4} and the W\"unschmann condition and do not give new conditions on the equation $(F)$ (we have checked it by direct computations in coordinates). The identity $4V'(K_2)+3V(K_2')=4V'(K_2)-3V(K_1)$ follows from \eqref{wun1_ord4}.
\end{proof}

\begin{corollary}\label{cor_ord4}
If a fourth order ODE admits a totally geodesic paraconformal connection then a solution $\beta$ to \eqref{eq_ord4} and \eqref{cond_ord4} is unique.
\end{corollary}
\begin{proof}
The one-form $\beta$ is uniquely defined by the condition $5b=-V(K_2)$.
\end{proof}

\begin{remark}
It would be nice to have a characterisation of $(F)$ admitting a totally geodesic paraconformal structure in terms of the curvature of the associated torsion-free Bryant connection. The curvature was explicitly computed in \cite{N2}. However, the Bryant connection is an object invariant with respect to the group of contact transformations which is much bigger than the group of point transformations. In fact a class of point equivalent ODEs splits into several classes of point equivalent ODEs. Therefore the problem would be to determine if a given class of contact equivalent ODEs contains a subclass of point equivalent ODEs admitting a totally geodesic paraconformal connection (a priori, the subclass is not unique). The problem is similar, in spirit, to the problem considered in \cite{DK2} where we characterise hyper-CR Einstein-Weyl structures in terms of point invariants. One gets invariants of very high order and a similar result should hold in the present case.
\end{remark}

\section{General case}\label{sec_gen}
In this section we consider an ODE in the form $(F)$. Our main result is as follows

\begin{theorem}\label{thm_gen}
A class of point equivalent ODEs of order $k+1$, $k\geq 4$, admits a totally geodesic paraconformal connection if and only if the W\"unschmann condition holds,
\begin{equation}\label{cartan_gen0}
V^{(i)}(K_{k-1})=0,\qquad i=0,\ldots,k-4,
\end{equation}
\begin{equation}\label{cartan_gen}
4V^{(k-2)}(K_{k-1})+3V^{(k-3)}(K_{k-1}')=0,
\end{equation}
and
\begin{eqnarray}
&&\left(\frac{2}{\gamma_k}-1\right)K_{k-1}V^{(k-3)}(K_{k-1})= (-1)^k\left(V^{(k-2)}(K_{k-1}')-2V^{(k-1)}(K_{k-1})\right),\label{cartan_gen2}\\
&&\left(\frac{2}{3\gamma_k}-1\right)K_{k-1}V^{(k-2)}(K_{k-1}) +\left(\frac{1}{\gamma_k}+\frac{k-3}{2}\right)K_{k-1}'V^{(k-3)}(K_{k-1})\nonumber\\ &&\qquad=(-1)^{k+1}\frac{1}{3}\left(V^{(k-2)}(K_{k-1}'')-V^{(k-1)}(K_{k-1}')+V^{(k)}(K_{k-1})\right),\label{cartan_gen3}
\end{eqnarray}
where $\gamma_k=-\frac{1}{2}\binom{k+2}{3}$. Additionally \eqref{cartan_gen} is equivalent to
$$
4V^{(k-2)}(K_{k-1})-\frac{6}{k-1}V^{(k-3)}(K_{k-2})=0.
$$
\end{theorem}
\begin{proof}
Let $\theta_0,\theta_1,\ldots,\theta_k$ be one-forms dual to vector fields $V,V',\ldots,V^{(k)}$. Assume that $\beta(V_k)=b$. Then $\beta=b\theta_k$. We compute that $\theta_k'=-\theta_{k-1}$, $\theta_k''=\theta_{k-2}+(-1)^{k+1}K_{k-1}\theta_k$ and $\theta_k'''=-\theta_{k-3}+(-1)^kK_{k-1}+(-1)^{k+1}(K_{k-1}'+K_{k-2})\theta_k$. Thus
$$
\beta'=b'\theta_k-b\theta_{k-1}
$$
and
\begin{eqnarray*}
&&\beta'''=\left(b'''+3(-1)^{k+1}b'K_{k-1}+(-1)^{k+1}b(K_{k-1}'+K_{k-2})\right)\theta_k+\\ &&\qquad\left(-3b''+(-1)^kK_{k-1}\right)\theta_{k-1}+3b'\theta_{k-2}-b\theta_{k-3}.
\end{eqnarray*}
Substituting this to \eqref{eq_gen}, evaluating on $V_0,\ldots,V_k$, and using the W\"unschmann condition \eqref{wun_gen} we get the conditions \eqref{cartan_gen0}-\eqref{cartan_gen2} in a way analogous to lower dimensional cases.
\end{proof}

\begin{corollary}\label{cor_ordgen}
If an ODE of order $k+1>4$ admits a totally geodesic paraconformal connection then a solution $\beta$ to \eqref{eq_gen} and \eqref{cond_gen} is unique.
\end{corollary}
\begin{proof}
The one-form $\beta$ is uniquely defined by the condition $\binom{k+2}{3}b=(-1)^k2V^{(k-3)}(K_{k-1})$ which is obtained by evaluation of \eqref{eq_gen} on $V^{(k-3)}$.
\end{proof}

\begin{remark}
Our conjecture is that the equations \eqref{cartan_gen2} and \eqref{cartan_gen3} are redundant and follow from \eqref{cartan_gen0}, \eqref{cartan_gen} and the W\"unschmann condition. However, we were unable to prove it in full generality.
\end{remark}

\section{Ricci curvature}\label{sec_ricci}
Let $\nabla$ be a totally geodesic paraconformal connection associated to $(F)$ of order $k+1\geq 3$. Direct computations show
$$
R(\nabla)(Y_1,Y_2)V=d\alpha(Y_1,Y_2)V+d\beta(Y_1,Y_2)V'-\beta\wedge\alpha'(Y_1,Y_2)V-\beta\wedge\beta'(Y_1,Y_2)V'
$$
and one gets
$$
Ric(\nabla)(V,V)=d\beta(V',V)-\beta\wedge\beta'(V',V).
$$
But, if $\beta(V)=\beta(V')=0$ then the right hand side vanishes (we use here $[V,V']\in\spn\{V,V'\}$) and therefore
\begin{equation}\label{eq_einstein}
Ric(\nabla)(V,V)=0.
\end{equation}
It follows that the symmetric part of the Ricci tensor of $\nabla$ is a section of the bundle of symmetric 2-tensors annihilating the field of null cones $x\mapsto C(x)$ of the paraconformal structure.

The bundle has rank one in the case of dimension 3. In fact it coincides with the conformal class $[\mathbf{g}]$. It follows that \eqref{eq_einstein} is equivalent to the Einstein-Weyl equation in this case.

The situation is more complicated in higher dimensions because the bundle of symmetric tensors annihilating the null cone has rank bigger than 1 (e.g. the rank is three in the case of dimension 4). It is an interesting question to determine if the condition \eqref{eq_einstein} implies that $\nabla$ is totally geodesic (it is the case in dimension 3).

\section{Veronese webs}\label{sec_Veronese}

A particularly simple example of paraconformal structures admitting totally geodesic connections can be obtained from special families of foliations, called \emph{Veronese webs}. We shall ultimately show that the structures are described by solutions to the integrable system \eqref{eq_int_sys}.

The Veronese webs are one-parameter families of foliations introduced by Gelfand and Zakharevich \cite{GZ} in connection to bi-Hamiltonian systems on odd-dimensional manifolds. Precisely, a one-parameter family of foliations $\{\FF_t\}_{t\in\R}$ of co-dimension 1 on a manifold $M$ of dimension $k+1$ is called Veronese web if any $x\in M$ has a neighbourhood $U$ such that there exists a co-frame $\omega_0,\ldots,\omega_k$ on $U$ such that
$$
T\FF_t=\ker\left(\omega_0+t\omega_1+\cdots+t^k\omega_k\right).
$$ 
In \cite{K2} we proved that there is a one to one correspondence between Veronese webs and ODEs for which all curvatures $K_i$ vanish. The equations are given modulo time-preserving contact transformations, mentioned earlier in Section \ref{sec_ODEinv}. But if all $K_i=0$ then automatically all conditions given in Theorems \ref{thm_ord2}, \ref{thm_ord3}, \ref{thm_ord4} and \ref{thm_gen} are satisfied. Therefore all Veronese webs admit totally geodesic paraconformal connections. The paraconformal structures obtained in this way will be referred to as \emph{of Veronese type}.

The structures are very specific. They correspond to projective structures defined by connections with skew-symmetric Ricci tensor in the case of order 2 (see \cite{K3}), and to Einstein-Weyl structures of hyper-CR type in the case of order 3 (see \cite{DK1}). The connections in the case of order 2 are projectively equivalent to the Chern connections of classical 3-webs \cite{K3}. The hyper-CR Einstein-Weyl structures are connected to integrable equations of hydrodynamic type and have Lax pairs with no terms in the direction of a spectral parameter \cite{D}. A characterisation of this special Einstein-Weyl structures in terms of point invariants of the related ODEs is complicated and involves 4 additional invariants of high order \cite{DK2}. In any case, Veronese webs define exactly those totally geodesic paraconformal structures for which the corresponding twistor space fibres over $\R P^1$ (c.f. \cite{D,DK1}).

Let $\omega(t)=\omega_0+t\omega_1+\cdots+t^k\omega_k$. Then the curve $t\mapsto \R\omega(t)\in P(T^*M)$ is a Veronese curve dual to the curve $t\mapsto\R V(t)\in P(TM)$
$$
V(t)=V_0+tV_1+\ldots+t^kV_k
$$
of null directions defining a paraconformal structure in Section \ref{sec_paraconf}. The duality means that
$$
V(t)\in\ker\omega(t)\cap\ker\omega'(t)\cap\ldots\cap\ker\omega^{(k-1)}(t)
$$
and conversely
$$
\ker\omega(t)=\VV_k(t)=\spn\{V(t),V'(t),\ldots,V^{(k-1)}(t)\}.
$$
In the case of Veronese webs the parameter $t$ is well defined globally uniquely modulo the M\"obius transformations $t\mapsto\frac{at+b}{ct+d}$, where $a,b,c,d\in\R$ and $ad-bc\neq 0$. Moreover, the distributions $\VV_k(t)$ are integrable for any particular choice of $t$. It means that
\begin{equation}\label{eq_integrability}
\omega(t)\wedge d\omega(t)=0,
\end{equation}
for any $t$. In order to get useful formulae we note that due to the integrability condition one can choose local coordinates $x_0,\ldots,x_k$ on $M$ such that $T\FF_{t_i}=\ker dx_i$ for some fixed $t_0,\ldots,t_k\in\R$. If we also assume that
$$
T\FF_{t_{k+1}}=\ker dw
$$
for a function $w=w(x_0,\ldots,x_k)$ then one verifies that
\begin{equation}\label{eq_omega}
\omega(t)=\sum_{i=0}^k (t_{k+1}-t_i)\prod_{j\neq i}(t-t_j)\partial_iwdx_i
\end{equation}
and $\omega$ is given up to a multiplication by a function on $M$. The following theorem for $k=2$ was proved in \cite{Z} and \cite{DK1}. It is new for $k>2$.

\begin{theorem}\label{thm3}
Let $M$ be a manifold of dimension $k+1$ with local coordinates $x_0,\ldots,x_k$ and let $t_0,\ldots,t_{k+1}\in\R$ be distinct numbers. Then, any $w$ satisfying the system
\begin{equation}\label{eq_hirota}
\sum_{cycl(i,j,l)}a_{ij,l}\partial_i\partial_j w\partial_l w=0,\qquad 0\leq i<j<l\leq k,
\end{equation}
where
$$
a_{ij,l}=(t_i-t_j)(t_{k+1}-t_l),
$$
defines a paraconformal structure via \eqref{eq_omega}, and conversely, any paraconformal structure of Veronese type can be locally put in this form. Moreover:
\begin{enumerate}
\item If $k>2$ then all totally geodesic paraconformal connections for a Veronese web are given by the formula
$$
\nabla\partial_i=\left(\frac{d(\partial_iw)}{\partial_iw}+\alpha\right)\partial_i
$$
where $\alpha$ is an arbitrary one-form on $M$. Moreover, there is the unique $\alpha$ such that the torsion of the corresponding paraconformal connection satisfies
$$
T(\nabla)(V(t),V'(t))\in\spn\{V(t)\},\qquad t\in\R.
$$
\item If $k=2$ then a Veronese web defines the following conformal metric
$$
\mathbf{g}=\sum_{i,j=0}^2(t_3-t_i)(t_3-t_j)\left(t_i^2+t_j^2-t_it_j-\sum_{l=0}^2t_l^2\right) \partial_iw\partial_jwdx_idx_j
$$
and there is the unique torsion-free totally geodesic paraconformal connection $\nabla$ such that $([\mathbf{g}], \nabla)$ is an Einstein-Weyl structure. The Weyl one-form for $\nabla$ is given by
$$
\varphi=\left(\frac{\partial_0\partial_1w}{\partial_1w}+\frac{\partial_0\partial_2w}{\partial_2w}\right)dx_0 +\left(\frac{\partial_0\partial_1w}{\partial_0w}+\frac{\partial_1\partial_2w}{\partial_2w}\right)dx_1 +\left(\frac{\partial_0\partial_2w}{\partial_0w}+\frac{\partial_1\partial_2w}{\partial_1w}\right)dx_2.
$$
\item If $k=1$ then for a given Veronese web there is the unique torsion-free parconformal connection satisfying the additional condition $\nabla_Y V(t)\in\{V(t)\}$ for any vector field $Y$ and any $t\in\R$. The connection is given by
$$
\nabla \partial_0=\left(\frac{\partial_0\partial_0w}{\partial_0w}-\frac{\partial_0\partial_1w}{\partial_1w}\right)dx_0\partial_0, \qquad\nabla \partial_1=\left(\frac{\partial_1\partial_1w}{\partial_1w}-\frac{\partial_0\partial_1w}{\partial_0w}\right)dx_1\partial_1.
$$
\end{enumerate}
\end{theorem}
\begin{proof}
The integrability condition \eqref{eq_integrability} written in coordinates and in terms of the function $w$ takes the form
$$
\sum_{i<j<l}\left(\sum_{cycl(i,j,l)}(T_i-T_j)T_l\partial_i\partial_j w\partial_lw\right) dx_i\wedge dx_j\wedge dx_l=0
$$
where
$$
T_i=(t_{k+1}-t_i)\prod_{j\neq i}(t-t_j).
$$
Thus, for any $i<j<l$ one gets the equation
\begin{equation}\label{eq_hirota0}
\sum_{cycl(i,j,l)}(T_i-T_j)T_l\partial_i\partial_j w\partial_l w=0
\end{equation}
which should be satisfied for any $t\in \R$. But the coefficient $(T_i-T_j)T_l$ equals to
$$
a_{ij,l}P_{ijl}(t,t_1,\ldots,t_k,t_{k+1})
$$
where
$$
P_{ijl}=(t-t_i)(t-t_j)(t-t_l)(t-t_{k+1})\prod_{s\neq i,j,l}(t-t_s)^2
$$
is a polynomial that does not depend on the permutation of indices $(i,j,l)$ and has zeroes (with multiplicities) exactly at points $t_0,\ldots,t_{k+1}$. Thus, for $t\in\{t_0,\ldots,t_{k+1}\}$  we get that the condition \eqref{eq_hirota0} is void and for $t\notin\{t_0,\ldots,t_{k+1}\}$ the condition \eqref{eq_hirota0} reduces to \eqref{eq_hirota}. 

The converse statement follows from the fact that any Veronese web can be written down as \eqref{eq_omega} in some coordinate system.

The formulae for $\nabla$ can be computed in the following way. We have $\nabla V(t)=\alpha V(t)+\beta V'(t)$ where the one-forms $\alpha$ and $\beta$, a priori, depend on $t$. However, in the Veronese case system \eqref{system} gives $\alpha'=-\frac{k}{2}\beta''$ and $\beta'''=0$. It follows that $\beta$ is a polynomial of degree 2 in $t$. Moreover, in the case of Veronese webs $V(t)$ considered on $M$ as in Section \ref{sec_paraconf} satisfies \eqref{cond_gen} and it means that $\beta=f\omega$ for some function $f$ on $M$. Hence, comparing the degrees of polynomials, we conclude $f=0$ for $k>2$. Consequently $\alpha$ does not depend on $t$. Therefore $\nabla V(t)=\alpha V(t)$ and in coordinates we get $\nabla\partial_i=\left(\frac{d(\partial_iw)}{\partial_iw}+\alpha\right)\partial_i$.

Now we shall prove that the condition $T(\nabla)(V(t),V'(t))\in\spn\{V(t)\}$ normalises $\alpha$ uniquely. First we prove that if $i<j$ then
\begin{equation}\label{eq_vv}
[V^{(i)}(t), V^{(j)}(t)]\in\spn\{V^{(i)}(t),V^{(i+1)}(t),\ldots, V^{(j)}(t)\}
\end{equation}
for any fixed $t$. This follows from the following
\begin{lemma}\label{lemma2}
All distributions $\VV_i(t)=\spn\{V(t),V'(t),\ldots,V^{(i-1)}(t)\}$, where $t\in\R$ is fixed, are integrable.
\end{lemma} 
\begin{proof}
Note that $\VV_i(t)=\ker\{\omega(t),\omega'(t),\ldots,\omega^{k-i}(t)\}$. Moreover, if $i=k$ then $\VV_k(t)$ is integrable by definition. This is expressed by \eqref{eq_integrability}. Now, we proceed by induction and prove that
$$
d\omega^{(i)}\wedge\omega\wedge\omega'\wedge\ldots\wedge\omega^{(i)}=0.
$$ 
Assuming that the formula above is true and differentiating it we get
$$
d\omega^{(i+1)}\wedge\omega\wedge\omega'\wedge\ldots\wedge\omega^{(i)}+
d\omega^{(i)}\wedge\omega\wedge\omega'\wedge\ldots\wedge\omega^{(i-1)}\wedge\omega^{(i+1)}=0.
$$
Finally, multiplying by $\omega^{(i+1)}$ the second term vanishes and we get
$$
d\omega^{(i+1)}\wedge\omega\wedge\omega'\wedge\ldots\wedge\omega^{(i+1)}=0.
$$ 
\end{proof}
Now, to prove \eqref{eq_vv} it is sufficient to consider $t=0$ only because $GL(2,\R)$ acts on $t$ and all points in the projective line are equally good. We have $\VV_i(0)=\spn\{V_0,\ldots,V_{i-1}\}$. But also we have $\VV_i(\infty)=\spn\{V_k,V_{k-1},\ldots,V_{k-i+1}\}$. Thus, applying Lemma \ref{lemma2} to $\VV_j(0)$ and $\VV_i(\infty)$, we get that the intersection $\VV_j(0)\cap\VV_i(\infty)$ is integrable and \eqref{eq_vv} holds.

Let us consider $T(\nabla)(V(t),V'(t))\mod V(t)$. We have
$$
T(\nabla)(V(t),V'(t))=\alpha(V(t))V'(t)-h(t)V'(t)\mod V(t) 
$$
where $h(t)$ is defined by $[V(t),V'(t)]=h(t)V'\mod V(t)$. Since $\alpha$ is independent of $t$ and $V(t)$ is a polynomial of degree $k$ in $t$ it is sufficient to show that $h(t)$ is a polynomial of degree $k$ in $t$. If it is the case, then $\alpha(V(t))=h(t)$ fixes $\alpha$ uniquely. But it is easy to show that \eqref{eq_vv} implies that $h(t)$ is a polynomial of degree $k$ in $t$. Indeed, we can normalise $\omega$ such that $\omega(t)(V^{(k)}(t))=1$. Then a simple induction gives $\omega^{(i)}(t)(V^{(j)}(t))=\pm\delta_i^{k-j}$ for all $i,j=0,\ldots,k$. In particular $h(t)=\pm \omega^{(k-1)}(t)([V(t),V'(t)])$. Differentiating this equation $k+1$ times, using \eqref{eq_vv} and the fact that $\omega$ and $V$ are polynomials of degree $k$ in $t$ we get that $h^{(k+1)}(t)=0$. This completes the proof in the case $k>2$.

In the case $k=2$ the theorem follows from \cite{DK1}. Here, we present a sketch of a different, more direct, proof. We have $\beta=f\omega$ and $\alpha=\tilde\alpha-ftk\omega''$, where $\tilde\alpha$ is a one-form on $M$. Simple but long computations prove that $\nabla$ is torsion-free if and only if
$$
f=\frac{1}{4}\left(\frac{\partial_0\partial_1w}{\partial_0w\partial_1w}- \frac{\partial_1\partial_2w}{\partial_1w\partial_2w}\right)\frac{1}{(t_3-t_1)(t_0-t_2)}
$$
and
$$
\tilde\alpha=-\frac{1}{4}\sum_{cycl(0,1,2)}\left( \frac{t_1-3t_2}{t_1-t_2}\frac{\partial_0\partial_1w}{\partial_1w} +\frac{t_2-3t_1}{t_2-t_1}\frac{\partial_0\partial_2w}{\partial_2w}\right)dx_0.
$$
The formula for $f$ does not depend on the permutation of indices $(0,1,2)$ due to \eqref{eq_hirota}. Having $f$ and $\tilde\alpha$ one has all ingredients necessary for the computation of $\nabla\mathbf{g}$ and consequently $\varphi$. The so-obtained connection satisfies the Einstein-Weyl equation due to results of Cartan.

In the case $k=1$ the one-form $\beta$ is linear in $t$ and equals $f\omega$ for some function $f$. It follows that $\alpha$ does not depend on $t$ and the vanishing of the torsion gives
$$
\alpha=-\left(\frac{\partial_0\partial_1w}{\partial_1w}-2f(t_2-t_0)\partial_0w\right)dx_0 -\left(\frac{\partial_0\partial_1w}{\partial_0w}-2f(t_2-t_1)\partial_1w\right)dx_1.
$$
It can be shown that the so obtained connection satisfies $\nabla_Y V(t)\in\spn\{V(t)\}$ for any $Y$ if and only if $f=0$. For $f\neq 0$ we only have $\nabla_{V(t)} V(t)\in\spn\{V(t)\}$. 
\end{proof}

We shall say that the unique connection from Theorem \ref{thm3} for $k>2$ is \emph{canonical}. Note that the unique connection in the case $k=1$ fits into the scheme $k>2$. On the other hand the unique torsion-free Weyl connection in the case $k=2$ is different since, in general, the associated one-form $\beta$ is non-trivial. However, in the case $k=2$ one can consider connections such that $\beta=0$ as well, and among them there is a unique one with the torsion normalised as in the case $k>2$. Another approach to canonical connections for Veronese webs has been recently proposed by A.~Panasyuk (personal communication).

The connection for $k=1$ is exactly the Chern connection of a 3-web \cite{K2}. The connection for $k=2$ is exactly the hyper-CR connection from \cite{DK1} (in \cite{DK1} the conformal class is defined by our $\mathbf{g}$ multiplied by $(w_0w_1w_2)^{-1}$ and consequently the one-form $\phi$ from \cite{DK1} equals $\varphi-d\ln(w_0w_1w_2)$). If $k>2$ we get the following result, which in particular shows that the totally geodesic paraconformal connections are not the torsion free connections used in \cite{B}.
\begin{corollary}\label{cor_webs}
If $k>2$ then for a non-flat paraconformal structure of Veronese type all totally geodesic paraconformal connections have non-vanishing torsion. In particular a Veronese web is flat if and only if the corresponding canonical connection is torsion-free.
\end{corollary}
\begin{proof}
We assume that a Veronese web $\{\FF_t\}_{t\in\R}$ is described by a function $w$ as in Theorem \ref{thm3}. A connection $\nabla$ is defined by the formula $\nabla\partial_i=\left(\frac{d(\partial_iw)}{\partial_iw}+\alpha\right)\partial_i
$. Thus, if the torsion of $\nabla$ vanishes then $\alpha(\partial_i)=-\frac{\partial_i\partial_jw}{\partial_jw}$ for any $j\neq i$. Taking $l\neq j$ and computing $\alpha(\partial_i)$ in two ways we get that $\partial_i\left(\frac{\partial_jw}{\partial_l w}\right)=0$. It implies that $\mathbf{grad}(w)$ is proportional to a vector field $(b_0,\ldots,b_k)$ where $b_i$ is a function of $x_i$ only. All $b_i$ are non-vanishing functions because any two foliations from the family $\{\FF_t\}_{t\in\R}$ intersect transversally. If we change local coordinates $\tilde x_i:=\int b_idx_i$ then still $\ker d\tilde x_i=T\FF_{t_i}$ and we get that in new coordinates the web is described by $w=\tilde x_0+\ldots+\tilde x_k$ which means that the corresponding paraconformal structure is flat.
\end{proof}

\paragraph{The Bryant connection.}
As mentioned before, there is a unique torsion-free paraconformal connection (Bryant connection) in the case of paraconformal structures defined by equations of order 4. It follows from above that in the case of Veronese webs the one-form $\beta$, involving the torsion-free connection via $\nabla V(t)=\alpha V(t)+\beta V'(t)$ does not vanish unless the structure is flat. Precisely, $\beta=\beta_0+t\beta_1+t^2\beta_2$ for some one-forms $\beta_i$ which do not depend on $t$. The one-forms $\beta_i$ can be computed explicitly. We shall do this in order to show the difference between totally geodesic connections and the Bryant connection.

Assume that a paraconformal structure is given by $V(t)=V_0+ tV_1(t)+t^2V_2(t)+t^3V_3(t)$ and let us introduce structural functions $c_{ij}^l$ by
$$
[V_i,V_j]=\sum_{l=0}^3c_{ij}^lV_l
$$
and let $\eta_0,\eta_1,\eta_2,\eta_3$ be the dual one-forms such that $\eta_i(V_j)=\delta_{ij}$. Then
\begin{eqnarray*}
&&\beta_0=\frac{1}{3}c_{02}^3\eta_0+\frac{1}{3}c_{12}^3\eta_1+(2c_{03}^2-c_{02}^1)\eta_2-c_{03}^1\eta_3,\\
&&\beta_1=(c_{03}^3-c_{02}^2)\eta_0+ \left(\frac{1}{3}c_{01}^0+\frac{1}{3}c_{13}^3-c_{03}^2\right)\eta_1+ \left(\frac{1}{3}c_{23}^3+\frac{1}{3}c_{02}^0-c_{03}^1\right)\eta_2+(c_{03}^0-c_{13}^1)\eta_3,\\
&&\beta_2=-c_{03}^2\eta_0+(2c_{03}^1-c_{13}^2)\eta_1+\frac{1}{3}c_{12}^0\eta_2+\frac{1}{3}c_{13}^0\eta_3,
\end{eqnarray*}
and additionally
\begin{eqnarray*}
\alpha=(3c_{02}^2-2c_{03}^3)\eta_0+(3c_{03}^2-c_{01}^0)\eta_1-c_{02}^0\eta_2-c_{03}^0\eta_3.
\end{eqnarray*} 
To get these expressions one considers $\nabla V(t)=\alpha V(t)+\beta V'(t)$ which gives $\nabla V_i$ in terms of $\alpha$ and $\beta_i$. Then the vanishing of the torsion gives 24 linear equations for 16 unknown functions: $\beta_i(V_j)$ and $\alpha(V_j)$, $i=0,1,2$, $j=0,\ldots,3$. However, the structural functions $c_{ij}^l$ satisfy eight additional linear relations given explicitly in \cite{K1}. The additional relations are exactly obstructions for the vanishing of the torsion. 

In the case of Veronese webs
$$
V_i=(-1)^{i+1}\binom{3}{i}\sum_{j=0}^3\frac{t_j^{3-i}}{\partial_j w}\left(\prod_{l\in\{0,1,2,3,4\}\setminus j}\frac{1}{t_l-t_j}\right)\partial_j
$$
and
$$
c_{ij}^l=\sum_{a,b=0}^3 d_{ij,ab}^l \frac{\partial_a\partial_b w}{\partial_a w\partial_b w}
$$
where $d_{ij,ab}^l$ are certain constants depending on $t_i$'s.

\paragraph{The Zakharevich conjecture.} 
The proof of Theorem \ref{thm3} gives also an elementary proof of the following result which was previously proven by Panasyuk \cite{P} in the analytic category and further generalised in \cite{BD} (the problem is called the Zakharevich conjecture in \cite{P}).
\begin{corollary}\label{cor_webs2}
The integrability condition \eqref{eq_integrability} is satisfied for any $t$ if and only if it is satisfied for $k+3$ distinct values of $t$.
\end{corollary}
\begin{proof}
If $t_{k+2}\notin\{t_0,\ldots,t_{k+1}\}$ then $P_{ijl}(t_{k+2})\neq 0$ and we can divide \eqref{eq_hirota0} by $P_{ijl}(t_{k+2})$. Then we get that \eqref{eq_hirota} has to be satisfied if the integrability condition holds for $t_{k+2}$. But then it follows from Theorem \ref{thm3} that the integrability condition holds for any $t$.
\end{proof}

\paragraph{Integrable systems, bi-Hamiltonian systems and Lax tuples. }
Note that the constants $a_{ij,l}$ in Theorem \ref{thm3} satisfy $\sum_{cycl(i,j,l)}a_{ij,l}=0$. Thus \eqref{eq_hirota} the equation is the dispersionless Hirota equation in the case $k=2$ (c.f. \cite{DK1,FK,Z}). It follows that the first part of Theorem \ref{thm3} is a generalisation of the preprint \cite[Corollary 3.7]{Z}.

Let us define the following $t$-dependent vector fields
$$
L_i(t)=-\frac{\partial_iw}{\partial_0w}\partial_0+a_i(t-t_i)\partial_i
$$
where $a_i=\frac{t_{k+1}-t_0}{t_{k+1}-t_i}$, for $i=1,\ldots,k$. The vector fields are linear in $t$ and satisfy
$\omega(t)(L_i(t))=0$. Therefore they span the distribution tangent to $\FF_t$, for any $t\in\R$. We will call $(L_1(t),\ldots, L_k(t))$ the \emph{Lax tuple} of the paraconformal structure of a Veronese type. One verifies that the Lax tuple commutes, i.e.
$$
[L_i,L_j]=0,
$$
if and only if \eqref{eq_int_sys} is satisfied. Thus, \eqref{eq_int_sys} is equivalent to \eqref{eq_hirota}. In this way we recover the hierarchy of integrable systems \cite[Equation 6]{DK1}. The following result extends \cite[Theorem 4.1]{DK1} to the case of arbitrary $k$.

\begin{proposition}
Let $N=M\times \R^k$ and let $y_1,\ldots,y_k$ be coordinates on $\R^k$. If $(L_i(t))_{i=}^{k}$ is a Lax tuple on $M$ defined by a Veronese web then 
$$
t\mapsto \sum_{i=1}^{k}L_i(t)\wedge\partial_{y_i}
$$
is a bi-Hamiltonian structure on $N$ and the construction is converse to the Gelfand-Zakharevich reduction.
\end{proposition}

\appendix

\section{Appendix: formulae}\label{ap_formulae}
All formulae below are either taken from \cite{JK} or computed by hands using \cite[Proposition 2.9]{JK}. All vector fields $V^{(i)}$ are given up to a multiplicative factor $g$ from equation \eqref{eq_g} which can be neglected.

\vskip 2ex
{Order 2:}
$$
K_0=-\partial_0F+\frac{1}{2}X_F(\partial_1F)-\frac{1}{4}(\partial_1F)^2
$$

$$
V=\partial_1,\qquad V'=-\partial_0-\frac{1}{2}\partial_1F\partial_1.
$$

\vskip 2ex
{Order 3:}
\begin{eqnarray*}
K_0&=&\partial_0F-X_F(\partial_1F) +\frac{1}{3}\partial_1F\partial_2F+ \frac{2}{3}X_F^2(\partial_2F)- \\
 &&\frac{2}{3}X_F(\partial_2F)\partial_2F+ \frac{2}{27}(\partial_2F)^3,\\
K_1&=&\partial_1F-X_F(\partial_2F)+\frac{1}{3}(\partial_2F)^2.
\end{eqnarray*}
\begin{eqnarray*}
V&=&\partial_2,\\
V'&=&-\partial_1-\frac{2}{3}\partial_2F\partial_2,\\
V''&=&\partial_0+\frac{1}{3}\partial_2F\partial_1+ \left(\partial_1F+\frac{4}{9}(\partial_2F)^2 -\frac{2}{3}X_F(\partial_2F)\right)\partial_2.
\end{eqnarray*}

\vskip 2ex
{Order 4:}
\begin{eqnarray*}
K_0&=&-\partial_0F+X_F(\partial_1F)-X_F^2(\partial_2F)+\frac{3}{4}X_F^3(\partial_3F) -\frac{9}{16}X_F(\partial_3F)^2+\\
&&\frac{18}{64}X_F(\partial_3F)(\partial_3F)^2- \frac{3}{256}(\partial_3F)^4-\frac{1}{4}\partial_1F\partial_3F +\frac{1}{2}X_F(\partial_2F)\partial_3F -\\
&&\frac{3}{4}X_F^2(\partial_3F)\partial_3F+ \frac{1}{4}X_F(\partial_3F)\partial_2F-\frac{1}{16}\partial_2F(\partial_3F)^2,\\
K_1&=&-\partial_1F+2X_F(\partial_2F)-2X_F^2(\partial_3F)-\frac{1}{2}\partial_2F\partial_3F +\frac{3}{2}X(\partial_3F)\partial_3F-\frac{1}{8}(\partial_3F)^3,\\
K_2&=&-\partial_2F+\frac{3}{2}X_F(\partial_3F)-\frac{3}{8}(\partial_3F)^2.
\end{eqnarray*}
\begin{eqnarray*}
V&=&\partial_3,\\
V'&=&-\partial_2-\frac{3}{4}\partial_3F\partial_3,\\
V''&=&\partial_1+\frac{1}{2}\partial_3F\partial_2+\left(\frac{9}{16}(\partial_3F)^2 -\frac{3}{4}X_F(\partial_3F)\right)\partial_3,\\
V'''&=&-\partial_0 -\frac{1}{4}\partial_3F\partial_1 +\left(\frac{5}{4}X_F(\partial_3F)-\frac{7}{16}(\partial_3F)^2\right)\partial_2+\\
&&\left(\frac{27}{16}X_F(\partial_3F)-\frac{3}{4}X_F^2(\partial_3F)-\partial_1F -\frac{1}{2}\partial_2F\partial_3F -\frac{27}{64}(\partial_3F)^3\right)\partial_3.
\end{eqnarray*}

\vskip 2ex
{General case, equations of order $k+1$:}
$$
K_{k-1}=(-1)^k\left(\partial_{k-1}F-\frac{k}{2}X_F(\partial_kF)+\frac{k}{2(k+1)}(\partial_kF)^2\right).
$$
$$
V=\partial_k,\qquad V'=-\partial_{k-1}-\frac{k}{k+1}\partial_kF\partial_k,
$$
$$
V^{(i)}=\sum_{j=0}^i\binom{i}{j}X_F^j(g)\ad_{X_F}^{i-j}\partial_k.
$$
where $X_F^j(g)$ can be computed using \eqref{eq_g} several times.


\paragraph{Acknowledgements.} The work has been partially supported by the Polish National Science Centre grant DEC-2011/03/D/ST1/03902.

\end{document}